\documentclass[a4paper]{article}

\usepackage{amsmath}
\usepackage{fullpage}
\usepackage{amsthm}
\usepackage{tikz}
\usepackage{tkz-graph}
\usepackage[ruled]{algorithm2e}
\usepackage{enumerate}
\usepackage{amsfonts}
\usepackage{color}

\newtheorem{theorem}{Theorem}
\newtheorem{lemma}[theorem]{Lemma}
\newtheorem{corollary}[theorem]{Corollary}

\usepackage{lineno,hyperref}
\modulolinenumbers[5]

\DeclareMathAlphabet{\mathpzc}{OT1}{pzc}{m}{it}
\newcommand{\Supp}[1]{\operatorname{\mathpzc{Supp}}(#1)}
\newcommand{\Core}[1]{\operatorname{\mathpzc{Core}}(#1)}
\newcommand{\Null}[1]{\operatorname{\mathpzc{Null}}(#1)}
\newcommand{\Rank}[1]{\operatorname{\mathpzc{Rank}}(#1)}
\newcommand{\Span}[1]{\operatorname{\mathpzc{Span}}(#1)}

\begin{document}
\title{On the structure of the fundamental subspaces of  (not necessarily symmetrical)  tree-patterned matrices}
\author{Daniel A.\ Jaume\footnote{daniel.jaume.tag@gmail.com} 
and
Adri\'{a}n Pastine\footnote{adrian.pastine.tag@gmail.com}\\
Universidad Nacional de San Luis\\
Departamento de Matem\'{a}tica\\
San Luis, Argentina}

\maketitle

\begin{abstract} A matrix is called acyclic if replacing the diagonal entries with $0$, and the nonzero diagonal entries with $1$, yields the adjacency matrix of a forest. In this paper we show that null space and the rank of a acyclic matrix with $0$ in the diagonal is obtained from the null space and the rank of the adjacency matrix of the forest by multipliying by non-singular diagonal matrices. We combine these methods with an algorithm for finding a sparsest basis of the null space of a forest to provide an optimal time algorithm for finding a  sparsest basis of the null space of acyclic matrices with $0$ in the diagonal.\end{abstract}


\section{Introduction}
Throughout this article, all graphs are assumed to be finite, undirected and without loops or multiple edges.
The vertices of a graph $G$ are denoted by $V(G)$ and its edges by $E(G)$.
We also assume that $\mathbb{F}$ denotes an arbitrary field. 
Following the notation in \cite{Mohammadian}, we denote by $\mathcal{M}_{\mathbb{F}}(G)$ 
the set of all matrices $M$ over $\mathbb{F}$ with rows and columns indexed by $V(G)$, so that
for every two distinct vertices $u,v\in V(G)$, the $(u,v)-$entry of $M$ is non-zero if and only if $\{u,v\}\in E(G)$.
Notice that the diagonal entries are allowed to be non-zero.
A matrix $M$ over $\mathbb{F}$ is said to be \textit{acyclic} if $M\in \mathcal{M}_{\mathbb{F}}(F)$  for a forest $F$.
If the forest $F$ is a tree, then $M$ is said to be \textit{tree-patterned}. 

Given a graph $G$, the adjacency matrix of $G$, denoted by $A(G)$, is a $(0,1)-$matrix in 
$\mathcal{M}_{\mathbb{F}}(G)$ with zero diagonal.

In \cite{Mohammadian}, the following lemma was proved.
\begin{lemma}\label{lemma symm}
Let $M$ be an acyclic matrix over a field $\mathbb{F}$. Then there exist a finite-dimensional
field extension $\mathbb{E}$ of $\mathbb{F}$ and a diagonal matrix $D$ over $\mathbb{E}$ such that
$D^{-1}MD$ is symmetric.
\end{lemma}
Due to a previous version of Lemma \ref{lemma symm} (which first appeared in \cite{Parter Youngs}), most of the study of acyclic matrices was done on symmetric matrices.
The matrices considered in this article do not need to be symmetric. This is done because the proofs
work almost identically, and in this way there is no need to calculate the necessary diagonal matrix
over the field extension. 

The fundamental spaces of a matrix $M$ are the null space, $\Null{M}$, and the rank, $\Rank{M}$.
The structure of the fundamental spaces of graph-patterned matrices 
has been studied in depth for symmetric tree-patterned 
matrices allowing non-zero entries in the diagonal, see for instance \cite{Fiedler, Nylen, Parter, Wiener}.
Most of these papers deal with the dimension of the null space, but none of them give a basis for it.

Given $M\in\mathcal{M}_{\mathbb{F}}(G)$, the \textit{null support} of $M$, denoted $\Supp{M}$, is the set of vertices of $G$ that have 
non-zero entries in at least one vector from $\Null{M}$. In other words, $v\in \Supp{M}$ if 
there is a vector $\overrightarrow{x}\in\Null{M}$ with $\overrightarrow{x}_v\neq 0$, 
where $\overrightarrow{x}_v$ is the coordinate of $\overrightarrow{x}$ corresponding to the vertex $v$.
In \cite{JM} the null support of adjacency matrices of forests has been studied in depth. The authors provided
a decomposition for any forest into an $S$-forest (the forests that have a unique maximum independent set), and an
$N$-forest (the forests that have a unique maximum matching). They showed that all the information of the
null space of $A(F)$ can be obtained from the $S$-forests and $N$-forests related to $F$. It was
implicitly shown
that $\Null{A(F)}$ coincides with the intersection of all the maximum independent sets of $F$.
In \cite{JMPS} an optimal time algorithm for finding a \emph{sparsest} (i.e., has the fewest nonzeros)
$\{-1,0,1\}$ basis for the null 
space of a forest has been found. It is important to notice that the problem of finding a sparsest basis of the null space of a matrix is an important problem for numerical applications, which is known to be NP-complete~\cite{MR857589} and even hard to approximate~\cite{MR3537025}. 

Let $\mathcal{M}_{\mathbb{F},0}(G)$ be the set of matrices in $\mathcal{M}_{\mathbb{F}}(G)$ with zero in the diagonal.
In Section \ref{Section Null} we show that given a forest $F$, the null space 
of any matrix in $\mathcal{M}_{\mathbb{F},0}(F)$
can be obtained by multiplying the null space of $F$ by a suitable non-singular diagonal matrix. This will
allow the use of all the tools developed in \cite{JM,JMPS} for the study of the null space of said matrices.
In particular we use the results of \cite{JMPS} to give an optimal time algorithm for finding a sparsest basis
of the null space of the matrix. In Section \ref{Section Rank} we prove that the rank
of any matrix in $\mathcal{M}_{\mathbb{F},0}(F)$ can be found by multiplying the rank of $F$ by a suitable non-singular
diagonal matrix.

The restriction to having zero in the diagonal may seem strong, but several problems (chemistry, electric 
conductance, flow in networks, etc) can be modeled with this kind of matrices.
\section{On the null space}\label{Section Null}
The null space of a graph is the direct sum of the null spaces of its connected components.
In a similar fashion, the null space of a matrix $M\in\mathcal{M}_{\mathbb{F}}(G)$ 
is the direct sum of the null spaces of $M$ over the connected components of $G$.
Because of this, we study the null space of matrices over trees and obtain results
for the null space of acyclic matrices.

Lemma \ref{suppindep}, which is fundamental for our results, first appeared in \cite{Mohammadian} as Theorem $8(i)$.

\begin{lemma}\cite{Mohammadian}\label{suppindep}
Let $F$ be a forest, and $M\in \mathcal{M}_{\mathbb{F},0}(F)$. 
If $\{v,w\}\in \Supp{M}$ then $\{v,w\}\not \in E(F)$.
\end{lemma}

Let $T$ be a tree, $M\in \mathcal{M}_{\mathbb{F},0}(T)$, and $v$ be a vertex of $T$.
For each vertex $w\in T$ let $vPw$ be the unique directed path from $v$ to $w$ in $T$.
In this sense $vPw$ and $wPv$ are different, because we care about the direction. 
We define the \textit{$v$-scalation of $M$} as the non-singular diagonal matrix with 
\[
D^{(M,v)}_{w,w}=\prod_{\substack{u\in\Supp{M},\\ (u,t)\in vPw}}M_{u,t}^{-1}\prod_{\substack{u\in \Supp{M},\\ (t,u)\in vPw}}M_{t,u}.
\]

The following lemma follows from the definition of $D^{(M,v)}$ and Lemma \ref{suppindep}
\begin{lemma}\label{cuentasD}
Let $T$ be a tree, $M\in \mathcal{M}_{\mathbb{F},0}(F)$, $u,v\in V(T)$, $vPu=(v=v_0,v_1,\ldots,u_1,u)$ and $vPu_2=(v=v_0,v_1,\ldots,u_1,u,u_2)$ two directed paths in $T$. The following statements are true.
\begin{enumerate}[i)]
\item If $u_1,u_2\in \Supp{M}$,
\[
D^{(M,v)}_{u_2,u_2}=D^{(M,v)}_{u_1,u_1}*M_{u,u_1}^{-1}*M_{u,u_2},
\]

\item if $u_1\in \Supp{M}$, $u_2\not\in \Supp{M}$,
\[
D^{(M,v)}_{u_2,u_2}=D^{(M,v)}_{u_1,u_1}*M_{u,u_1}^{-1},
\]

\item if $u_1\not\in \Supp{M}$, $u_2\in \Supp{M}$,
\[
D^{(M,v)}_{u_2,u_2}=D^{(M,v)}_{u_1,u_1}*M_{u,u_2},
\]

\item if $u_1,u,u_2\not\in \Supp{M}$,
\[
D^{(M,v)}_{u_2,u_2}=D^{(M,v)}_{u_1,u_1},
\]

\item if $u\in \Supp{M}$,
\[
D^{(M,v)}_{u_2,u_2}=D^{(M,v)}_{u_1,u_1}*M_{u_1,u}*M_{u,u_2}^{-1}.
\]
\end{enumerate}
\end{lemma}

As a direct consequence of Lemma \ref{cuentasD}, the matrices $D^{(M,v)}$ and $\left(D^{(M,v)}\right)^{-1}$ 
can be obtained in linear time over the number of vertices, as at most one multiplication must be done at each
vertex.

\begin{theorem}\label{main}
Given a tree $T$, a matrix $M\in \mathcal{M}_{\mathbb{F},0}(T)$ and a vertex $v\in\Supp{M}$, a vector $\overrightarrow{x}$ is in $\Null{M}$ if and only if $D^{(M,v)}\overrightarrow{x}$ is in $\Null{A(T)}$.
\end{theorem}
\begin{proof}
Suppose $\overrightarrow{x}=(x_1,\ldots,x_n)\in \Null{M}$, and let $v\in\Supp{M}$.

We want to show $A(T)D^{(M,v)}\overrightarrow{x}=0$. Consider  $\left(A(T)D^{(M,v)}\overrightarrow{x}\right)_{u}$. If $\overrightarrow{x}_{w}=0$ for all $w\sim u$, then 

\begin{align*}
\left(A(T)D^{(M,v)}\overrightarrow{x}\right)_{u}&=\sum_{w\sim u}D^{(M,v)}_{w,w}\overrightarrow{x}_{w}
=0.
\end{align*}

Otherwise, let $(w_1,u)\in vPu$. If $\overrightarrow{x}_{w_1}\neq 0$, then applying Lemma \ref{cuentasD} we get: 
\begin{align*}
(A(T)D\overrightarrow{x})_{u}&=D_{w_1,w_1}\overrightarrow{x}_{w_1}+ \sum_{\substack{w\sim u,\\ w\neq w_1}}D_{w,w}\overrightarrow{x}_{w}
\\
&=D_{w_1,w_1}\overrightarrow{x}_{w_1}+\sum_{\substack{w\sim u,\\ w\neq w_1}}\overrightarrow{x}_{w}D_{w_1,w_1}
\left(M_{u,w_1}^{-1}M_{u,w}\right)\\
&=D_{w_1,w_1}M_{u,w_1}^{-1}\left(M_{u,w_1}\overrightarrow{x}_{w_1}+\sum_{\substack{w\sim u,\\ w\neq w_1}}
\overrightarrow{x}_{w}M_{u,w}\right)\\
&=D_{w_1,w_1}M_{u,w_1}^{-1}(M\overrightarrow{x})_{u}\\
&=0
\end{align*}

If $\overrightarrow{x}_{w_1}=0$, and $\overrightarrow{x}_{w}\neq 0$ for some $w\sim u$, Lemma \ref{cuentasD} implies:
\begin{align*}
(A(T)D\overrightarrow{x})_{u}&= \sum_{\substack{w\sim u,\\ w\neq w_1}}D_{w,w}\overrightarrow{x}_{w}\\
&=\sum_{\substack{w\sim u,\\ w\neq w_1}}\overrightarrow{x}_{w}D_{w_1,w_1}M_{u,w}\\
&=D_{w_1,w_1}(M\overrightarrow{x})_{u}\\
&=0
\end{align*}

Therefore $A(T)D^{(M,v)}\overrightarrow{x}=0$.

For the reciprocal similar arguments working with $D^{-1}$ instead of $D$ yield the result.
\end{proof}

If we consider a forest instead of a tree, then a diagonal matrix can be obtained 
by choosing one vertex in each connected 
component, and working in a similar fashion. Let $F$ be a forest, and $U\subset V(F)$ such that $U$ has at most one 
vertex in each connected component of $F$. We define the \textit{$U$-scalation of $M$} as the non-singular diagonal 
matrix with 
\[
D^{(M,U)}_{w,w}=\prod_{v\in U} D^{(M,v)}_{w,w},
\]
where $D^{(M,v)}_{w,w}=1$ if $v$ and $w$ belong to different connected components of $F$.

In order to generalize Theorem \ref{main} we need to have a set $U$ that has elements in all the necessary
components of a forest.
Given a forest $F$ with connected components $T_1,...,T_k$, a set $U\subset \Supp{M}$ is 
\textit{supp-transversal of $M$} if 
$U\cap V(T_i)\neq \emptyset$ whenever $\Supp{M}\cap V(T_i)\neq \emptyset$.
We have the following.
\begin{corollary}
Given a forest $F$ and a matrix $M\in \mathcal{M}_{\mathbb{F},0}(F)$, a vector $\overrightarrow{x}$ is in $\Null{M}$ if 
and only if $D^{(M,U)}\overrightarrow{x}$ is in $\Null{A(F)}$ for every set $U$ supp-transversal of $M$.
\end{corollary}

\begin{corollary}\label{DADBnull}
Let $F$ be a forest, and $M,N\in \mathcal{M}_{\mathbb{F},0}(F)$. For every set $U_1$ supp-transversal of $M$  and every set $U_2$ supp-transversal of $N$, $\overrightarrow{x}\in \Null{M}$ if and only if 
$D^{(M,U_1)}(D^{(N,U_2)})^{-1}\overrightarrow{x}$ is in the nullspace of $N$.
\end{corollary}

\begin{corollary}
Given a forest $F$, and a pair of matrices $M,N\in \mathcal{M}_{\mathbb{F},0}(F)$, a vertex $v$ is in $\Supp{M}$ if and only if $v$ it is in $\Supp{N}$.
\end{corollary}
\begin{proof}
If $v\in\Supp{M}$, there is some $\overrightarrow{x}\in\Null{M}$ with $\overrightarrow{x}_v\neq 0$. Thus by Corollary 
\ref{DADBnull}, the vector $D^{(M,v_1)}(D^{(N,v_2)})^{-1}\overrightarrow{x}$ is in $\Null{N}$ for some  
$v_1\in \Supp{M}$, $v_2\in \Supp{N}$. But 
$\left(D^{(M,v_1)}(D^{(N,v_2)})^{-1}\overrightarrow{x}_v\right)$ is not $0$ because $D^{(M,v_1)}$ and $D^{(N,v_2)}$ are 
non-singular diagonal matrices. Therefore $v\in \Supp{N}$.
\end{proof}

The next two corollaries are given to illustrate the strength of Corollary \ref{DADBnull}, and
the relation between the structure of a forest $F$ and the null space of the matrices in 
$\mathcal{M}_{\mathbb{F},0}(F)$.
In \cite{JM}, the concept of the $S$-set of a tree $T$ was introduced. It is the subgraph induced by 
$\Supp{T}\cup N(\Supp{T})$ (where $N(\Supp{T})$ denotes the neighborhood of $\Supp{T}$)
and is denoted $\mathcal{F}_S(T)$. In other words, $\mathcal{F}_S(T)$ is the 
subgraph induced by the vertices in the null support and the 
neighbors of the vertices in the null support. One of their main results is the fact that the null space
of a tree $T$ is the same as the null space of $\mathcal{F}_S(T)$, extended with $0$ to match the dimensions.
The same holds true for matrices in $\mathcal{M}_{\mathbb{F},0}(T)$. And, by doing direct sum, the same holds
true for forests. 

To help with the cleanness of the next corollary, we introduce some notation.
Given a matrix $M\in\mathcal{M}_{\mathbb{F},0}(F)$, and $G$ an induced subgraph of $F$, we denote by $M[G]$
the matrix obtained by deleting the rows and columns of vertices not in $G$. We do the same for 
vectors, $\overrightarrow{x}[G]$ denotes the vector obtained from $\overrightarrow{x}$ by deleting the coordinates
correspoding to vertices not in $G$.
\begin{corollary}
Let $F$ be a forest, $M\in\mathcal{M}_{\mathbb{F},0}(F)$ and $\overrightarrow{x}\in \mathbb{F}^F$. 
Then $\overrightarrow{x}\in \Null{M}$ if and only if:
\begin{itemize}
\item $\overrightarrow{x}\left[\mathcal{F}_S(F)\right]\in \Null{M\left[\mathcal{F}_S(F)\right]}$, and
\item $\overrightarrow{x}\left[V(F)\setminus \mathcal{F}_S(F)\right]=\overrightarrow{0}$.
\end{itemize}
\end{corollary}

A helpful result, implicit in \cite{JM}, is the fact that the $\Supp{T}$ is the intersection
of all the maximum independent sets of $T$. Which yields the following.
\begin{corollary}
Let $F$ be a forest, $M\in\mathcal{M}_{\mathbb{F},0}(F)$, and $v\in V(F)$. Then $v\in \Supp{M}$ if and only
if $v$ is in every maximum independent set of $F$.
\end{corollary}

The next corollary, originally proved in \cite{Mohammadian}, follows directly from Corollary \ref{DADBnull} and
the fact that dimension of the rank of a tree is twice its matching number (see \cite{bevis}).
\begin{corollary}
If $F$ is a forest and $M\in\mathcal{M}_{\mathbb{F},0}(F)$, then $\dim{\Rank{M}}=2m$,
where $m$ is the size of a maximum matching in $F$.
\end{corollary}

One can now use the relation between the null space of a forest $F$ and the null space of any matrix 
$M\in \mathcal{M}_{\mathbb{F},0}(F)$ 
to find a basis for the null space of $M$, which is done in Algorithm \ref{algorithm}. 
Finding the forest $F$ given the matrix $M$ takes linear time,
because it can be obtained by replacing the entries by $1$, and the matrix has at most $2(n-1)$ nonzero entries
(the edges of the forest). As $D^{(M,U)}$ does not change the support of a vector, a sparsest basis for 
the null space of $F$ provides
a sparsest basis for the null space of $M$ once it is multiplied by $(D^{(M,U)})^{-1}$.
In \cite{JMPS} the support of a forest was found in linear time, and a $\{-1,0,1\}$ and sparsest basis for 
the null space of a forest was found in optimal time.

Using the support, finding $D^{(M,U)}$ and $(D^{(M,U)})^{-1}$ takes linear time on the number of vertices, 
as for each vertex only one operation needs to be done. 
Afterwards, multiplying the elements of the basis found using the algorithm from \cite{JMPS} by
$(D^{(M,U)}_{v,v})^{-1}$ takes one operation per each non-zero entry in the vectors of the basis of the forest. 
Hence a sparsest basis for the null space of $M$ can be found in optimal time.

\begin{algorithm}[H]\caption{for finding a sparsest basis of the null space a acyclic matrix with $0$ in the diagonal.}
 \label{algorithm}
		
		\begin{enumerate}
			\item INPUT: \(M\), a tree-patterned matrix with $0$ in the diagonal.
			\item Find $F$ such that $M\in \mathcal{M}_{\mathbb{F},0}(F)$.
			\item Apply the algorithms from \cite{JMPS} to find a sparsest basis, $\mathcal{B}_F$ of $A(F)$ and 
			$\Supp{F}$.
			\item Find the connected components of $F$.
			\item For each $T_i$ connected component of $F$ with $\Supp{A(T_i)}\neq \emptyset$ chose 
			$v_i\in \Supp{A(T_i)}$. 
			\item Let $U$ be the set of the chosen $v_i$ and calculate $(D^{(M,U)})^{-1}$.
			\item Calculate $\mathcal{B}_M=(D^{(M,U)})^{-1}\mathcal{B}_F$
			\item OUTPUT $\mathcal{B}_M$
		\end{enumerate}
	\end{algorithm}

Algorithm \ref{algorithm} is important because it expands on the set of matrices for which a sparsest basis
of the null space can be found in optimal.
\section{On the rank}\label{Section Rank}
In the previous section we proved that given a forest $F$ and $M\in \mathcal{M}_{\mathbb{F},0}(F)$,  $\Null{M}$ 
is a non-singular diagonal multiplication of $\Null{F}$. In this section show that  $\Rank{M}$ 
is a non-singular diagonal 
multiplication of $\Rank{F}$. In order to do so, first we find a basis for the rank of $M$.

Let $v\not\in \Supp{M}$, we define its \textit{supported-neighborhood vector}, $\overrightarrow{s}_v$, as 
\[
\overrightarrow{s}_v(M)=\sum_{w\in \Supp{M}\cap N(v)}M(v,w)\overrightarrow{e}_w,
\]
where $\overrightarrow{e}_w$ denotes the vector with $1$ in coordinate $w$ and $0$ elsewhere.

In \cite{JMS} it was shown that 
\[B(F):=\bigcup\limits_{v\not\in\Supp{F}}\{\overrightarrow{e}_v,\overrightarrow{s}_v(F)\}
\setminus\{\overrightarrow{0}\}\]
is a basis for the rank of $F$. We show the same result for $M\in \mathcal{M}_{\mathbb{F},0}(F)$.

\begin{lemma}
If $F$ is a forest and $M\in \mathcal{M}_{\mathbb{F},0}(F)$, then
\[
B(M):=\bigcup\limits_{v\not\in\Supp{M}}\{\overrightarrow{e}_v,\overrightarrow{s}_v(M)\}\setminus
\{\overrightarrow{0}\}
\]
is a basis for the rank of $M$.
\end{lemma}
\begin{proof}
It is easy to see that all columns of $M$ can be written as linear combinations of 
\[B(M)=\bigcup\limits_{v\not\in
\Supp{M}}\{\overrightarrow{e}_v,\overrightarrow{s}_v(M)\}\setminus\{\overrightarrow{0}\}.
\]
Hence $\Rank{M}\subset\Span{B(M)}$.

But $\dim(\Null{M})=\dim(\Null{F})$ by Theorem \ref{main}. Thus 
\[
\dim(\Rank{M})=\dim(\Rank{F})=|B(F)|=|B(M)|.\]

Therefore $B(M)=\bigcup\limits_{v\not\in\Supp{M}}\{\overrightarrow{e}_v,\overrightarrow{s}_v(M)\}\setminus
\{\overrightarrow{0}\}$ is a basis for the rank of $M$.
\end{proof}

Again, we work on a tree instead of a forest, because the rank is the direct 
sum of the ranks of the connected components.

Let $T$ be a tree, $M\in \mathcal{M}_{\mathbb{F},0}(T)$ and $v$ a vertex of $F$. For each vertex 
$w$ let $\pi(v,w)$ be second 
vertex in $vPw$, where $v$ is the first vertex of 
the path. We define the \textit{$v$-normalization of $M$} as the non-singular diagonal matrix 
with 
\[
C^{(M,v)}_{w,w}=M_{v,\pi(v,w)}.
\]

We define the \textit{rank-normalization of $M$}, $R^M$, as the product of $C^{(M,v)}$ over all vertices $v\not\in
\Supp{M}$.
\[
R^M=\prod_{v\not\in\Supp{M}}C^{(M,v)}
\]

Let $v\in T$, then we say that $v$ is a \textit{core vertex} of $M$ if 
$N(v)\cap\Supp{M}\neq \emptyset$. In other words, core vertices are the neighbors of 
vertices in the null support of $M$. The \textit{core} of $M$, $\Core{M}$ is the set 
of all core vertices of $M$.
\begin{lemma}
If $T$ is a tree and $M\in \mathcal{M}_{\mathbb{F},0}(T)$, then $R^M\Rank{T}=\Rank{M}$.
\end{lemma}
\begin{proof}
Let $v\not\in\Supp{M}$. We have
\begin{align*}
R^M\overrightarrow{s}_v(T)=&\sum_{w\in \Supp{M}\cap N(v)}R^M\overrightarrow{e}_{w}\\
=&\sum_{w\in \Supp{M}\cap N(v)}C^{(M,v)}\prod_{\substack{u\not\in\Supp{M},\\ u\neq v}}C^{(M,u)}\overrightarrow{e}_{w}.\\
\end{align*}
But if $w\in\Supp{M}\cap N(v)$ and $u\not\in\Supp{M}$ with $u\neq v$, then $\pi(u,w)=\pi(u,v)$. 
Notice that $C^{(M,u)}e_{w}=M(u,\pi(u,v))\overrightarrow{e}_{w}$. Hence
\begin{align*}
R^M\overrightarrow{s}_v(T)=&\sum_{w\in \Supp{M}\cap N(v)}C^{(M,v)}\prod_{\substack{u\not\in\Supp{M},\\ u\neq v}}C^{(M,u)}\overrightarrow{e}_{w}\\
=&\sum_{w\in \Supp{M}\cap N(v)}C^{(M,v)}\prod_{\substack{u\not\in\Supp{M},\\ u\neq v}}M(u,\pi(u,v))\overrightarrow{e}_{w}\\
=&\prod_{\substack{u\not\in\Supp{M},\\ u\neq v}}M(u,\pi(u,v))\sum_{w\in \Supp{M}\cap N(v)}C^{(M,v)}\overrightarrow{e}_{w}.\\
\end{align*}
On the other hand, if $w\in N(v)$, $C^{(M,v)}\overrightarrow{e}_{w}=M(v,w)\overrightarrow{e}_{w}$. Therefore
\begin{align*}
R^M\overrightarrow{s}_v(T)=&\prod_{\substack{u\not\in\Supp{M},\\ u\neq v}}M(u,\pi(u,v))\sum_{w\in \Supp{M}\cap N(v)}C^{(M,v)}\overrightarrow{e}_{w}\\
=&\prod_{\substack{u\not\in\Supp{M},\\ u\neq v}}M(u,\pi(u,v))\sum_{w\in \Supp{M}\cap N(v)}M(v,w)\overrightarrow{e}_{w}\\
=&\prod_{\substack{u\not\in\Supp{M},\\ u\neq v}}M(u,\pi(u,v))\overrightarrow{s}_v(M).
\end{align*}
Hence $\Span{R^M\overrightarrow{s}_v(T)}=\Span{\overrightarrow{s}_v(M)}$. It is easy to see that $\Span{ R^M\overrightarrow{e}_v}=\Span{\overrightarrow{e}_v}$. Therefore $R^M\Rank{T}=\Rank{M}$.
\end{proof}

If instead we consider a forest, then a diagonal matrix can be obtained by having $C^{(M,v)}_{w,w}=1$ when $v$ and $w$ are in different connected components. Hence, we have the following.
\begin{corollary}
Given a tree $F$ and a matrix $M\in \mathcal{M}_{\mathbb{F},0}(F)$, $R^M\Rank{F}=\Rank{M}$.
\end{corollary}

The following result is a direct application, because $R^M$ is nonsingular.
\begin{corollary}
Let $F$ be a forest, and $M,N\in \mathcal{M}_{\mathbb{F},0}(F)$. The vector
$\overrightarrow{x}$ is in $\Rank{M}$ if and only if the vector
$R^N(R^M)^{-1}\overrightarrow{x}$ is in $\Rank{N}$.
\end{corollary}
\section{Conclusion}
There is a strong relation between the rank and the null space of a tree-patterned (acyclic)
matrix with 
diagonal $0$, and its underlying tree (forest).
It would be interesting to study what happens when non-zero diagonal entries are allowed, or when a different graph
is used. We conjecture that there will still be a strong relation, but it will not be so straightforward.
For example, having non-zero diagonal entries only in the vertices in $\Core{F}$ should have no effect in the
null space of the matrix.

\end{document}